\documentclass[a4paper]{amsart}

\usepackage{amsmath,amsthm,amssymb}
\usepackage{hyperref}
\usepackage{todonotes}
\usepackage[all]{xy}

\newtheorem{theorem}{Theorem}[section]
\newtheorem{proposition}[theorem]{Proposition}
\newtheorem{lemma}[theorem]{Lemma}

\newtheorem{remark}[theorem]{Remark}

\theoremstyle{definition}

\newtheorem{definition}[theorem]{Definition}

\newcommand{\pullbackcorner}[1][dr]{\save*!/#1-1.2pc/#1:(-1,1)@^{|-}\restore}

\newcommand{\cat}[1]{\mathbb{#1}}
\newcommand{\catc}{\cat{C}}
\newcommand{\set}{\mathbf{Set}}

\newcommand{\opcat}[1]{{#1}^\mathrm{op}}

\newcommand{\smcat}[1]{\mathcal{#1}}
\newcommand{\cod}{\operatorname{cod}}

\newcommand{\intv}{\mathbb{I}}
\newcommand{\fintv}{\tilde{\mathbb{I}}}
\newcommand{\nat}{\mathbb{N}}

\newcommand{\elu}{\mathrm{El}}
\newcommand{\elv}{\mathrm{El}}
\newcommand{\hequiv}{\operatorname{Equiv}}
\newcommand{\yoneda}{\mathbf{y}}
\newcommand{\preshf}[1]{\set^{\opcat{#1}}}

\newcommand{\izf}{\mathbf{IZF}}
\newcommand{\czf}{\mathbf{CZF}}
\newcommand{\llpo}{\mathbf{LLPO}}
\newcommand{\inacc}{\mathbf{Inacc}}

\newcommand{\klone}{\mathcal{K}_1}
\newcommand{\kltwo}{\mathcal{K}_2}

\newcommand{\pcaa}{\mathcal{A}}
\newcommand{\rt}{\mathbf{RT}}

\newcommand{\kv}{\mathcal{KV}}
\newcommand{\mc}{\mathbf{Mc}}

\newcommand{\refl}{\mathtt{refl}}

\newcommand{\iscontr}{\operatorname{IsContr}}

\title[Separating Path and Identity Types in Presheaves]
  {Separating Path and Identity Types in
  Presheaf Models of Univalent Type Theory}

\author{Andrew W Swan}

\begin{document}
\maketitle

\begin{abstract}
  We give a collection of results regarding path types, identity types
  and univalent universes in certain models of type theory based on
  presheaves.
  
  The main result is that path types cannot be used directly as
  identity types in any Orton-Pitts style model of univalent type
  theory with propositional truncation in presheaf assemblies over the
  first and second Kleene algebras.

  We also give a Brouwerian counterexample showing that there is no
  constructive proof that there is an Orton-Pitts model of type theory
  in presheaves when the universe is based on a standard construction
  due to Hofmann and Streicher, and path types are identity types. A
  similar proof shows that path types are not identity types in
  internal presheaves in realizability toposes as long as a certain
  universe can be extended to a univalent one.

  We show that one of our key lemmas has a purely syntactic variant in
  intensional type theory and use it to make some minor but curious
  observations on the behaviour of cofibrations in syntactic
  categories.
\end{abstract}

\section{Introduction}
\label{sec:intro}

In the cubical set model of homotopy type theory, and more generally
in members of the classes of models considered by Gambino and Sattler,
by Van den Berg and Frumin, and by Orton and Pitts, the most basic
notion of identity is that of path type. In that construction one uses
exponentiation with an interval object $\intv$. Given an object $X$,
we call $X^\intv$ the \emph{path type on $X$} and think of it as the
collection of paths between two elements of $X$. This can then be used
to produce path objects for any fibration $X \to Y$ via the mapping
path factorisation.

In order to interpret identity types in type theory as path types it
is necessary to show that the reflexivity (or ``constant paths'') map
$r^X \colon X \to X^\intv$ is a trivial cofibration. In many
natural examples however, it is difficult to show that this map
is a trivial cofibration, or even
just a cofibration. For example, in \cite{bchcubicalsets}, Bezem, Coquand
and Huber gave a definition of path type, but did not show how to
interpret identity types that strictly satisfy the $J$-computation
rule. In \cite{swannomawfs} the author gave both an explanation for
why this was difficult constructively as well as a solution. The
explanation was that the definition of fibration used in the BCH model
leads to an awfs where the trivial cofibrations always have pointwise
decidable image. However, when $X$ is the nerve of a complete metric
space, the map $r^X$ is essentially the inclusion of constant paths into
the set of all paths in the usual topological sense. One then gives a
Brouwerian counterexample to show that there is no constructive proof
that such inclusions are pointwise decidable. The solution was to use a
second, more elaborate construction to obtain identity types that do
satisfy the $J$-computation rule.

Of course, one way to prevent this Brouwerian counterexample from
causing problems is to use a different definition of fibration for
interpreting types. From an abstract point of view this can be
achieved by taking any monomorphism to be a cofibration, by
definition. The reflexivity map $r^X \colon X \to X^\intv$ is in
general a (split) monomorphism and so a cofibration. From here one can
show that it is in fact a trivial cofibration and thereby use it to
implement identity types. From a syntactic point of view it can be
achieved by adding so called \emph{regularity} or \emph{normality}
conditions to the definition of Kan operations, which state that
composition along a degenerate open box is the identity. Indeed, at
one point in the development of cubical type theory, Coquand and
collaborators did use such a regularity condition. Although this was
never published, it was in many ways very successful and did lead to a
version of type theory where path types are identity types, with some
higher inductive types and is believed to be consistent and have good
computational properties. The only problem with this approach, as
discovered by Dan Licata, is that it is completely unclear how to
construct a universe satisfying univalence. See
\cite{licataregdiscussion} for some informal discussion of this issue
online.

The aim of this paper is to give a wide class of counterexamples that
apply not just to one model, but to a range of different categories with
varying definitions of cofibration and thereby also varying
definitions of fibration. We will look at a class of structures based
on models of type theory in presheaves and in particular presheaf
assemblies (for instance cubical assemblies as defined by Uemura in
\cite{uemuracubasm}). We develop two basic techniques.
\begin{enumerate}
\item Using the assumption that path types are already identity types
  to show that certain maps have to be cofibrations.
\item Using a univalent universe to show that certain cofibrations are
  pointwise stable under double negation.
\end{enumerate}

We will then combine these to derive statements that are non
constructive and in realizability models outright false.

\subsection*{Acknowledgements}
\label{sec:acknowledgements}

I'm grateful for useful discussions and suggestions relating to this
topic from Benno van den Berg, Martijn den Besten, Thierry Coquand,
Nicola Gambino, Simon Huber, Peter Lumsdaine, Ian Orton, Andy Pitts,
Christian Sattler and Taichi Uemura.

\section{General Set Up}
\label{sec:general-set-up}

We work over a setting based on a definition due to Van den Berg and
Frumin in \cite{vdbergfrumin}, in turn based on a definition due to
Gambino and Sattler in \cite{gambinosattlerpi}. We assume that the
reader is already familiar with the notions used there such as
left/right lifting problems, wfs's and pushout products.  The Van den
Berg-Frumin definition can also be seen as a purely homotopical
reformulation of the Orton-Pitts axioms in
\cite{pittsortoncubtopos}. See e.g. \cite[Section 7.5.2]{swanliftprob}
for a discussion of the precise relation between these two
approaches. We weaken the Van den Berg-Frumin definition in a few
ways. Most importantly, we weaken the requirement that the underlying
category is a topos to locally cartesian closed category with finite
colimits. This is necessary to even include presheaf assemblies as an
example. We drop the requirement that cofibrations are classified by a
single universal cofibration, which is the object $\mathtt{Cof}$ in
the Orton-Pitts formulation, and $\Sigma$ in the the Van den
Berg-Frumin formulation.

We also have no strictness condition (Orton and Pitts' axiom
$\mathtt{ax}_9$), although there is an important point here. We will
derive some statements very much related to strictness from the
assumption that a univalent universe exists.

Essentially we consider locally cartesian closed categories with a
good notion of cofibration and interval object, where cofibrations are
closed under pullback and finite union, and generating trivial
cofibrations are given by the pushout product of a cofibration with an
endpoint inclusion. Formally, we state this as
follows.

Let $\catc$ be a locally cartesian closed category with all finite
colimits. We assume we are also given a class of maps whose elements
we call \emph{cofibrations} and an interval object
$\delta_0, \delta_1 \colon 1 \to \intv$. We use these to define the
following classes of maps.
\begin{definition}
  \begin{enumerate}
  \item We say a map is a \emph{trivial fibration} if it has the right
    lifting property against every cofibration.
  \item We say a map is a \emph{fibration}, if it has the right
    lifting property against $\delta_0 \hat{\times} m$ and $\delta_1
    \hat{\times} m$ for every cofibration $m$.
  \item We say an object $X$ is \emph{fibrant} if the unique map $X
    \to 1$ is a fibration.
  \item We say an object $X$ is \emph{cofibrant} if the unique map $0
    \to X$ is a cofibration.
  \item We say a map is a \emph{trivial cofibration} if it has the
    left lifting property against every fibration.
  \end{enumerate}
\end{definition}

We assume throughout that all of the following conditions are
satisfied.
\begin{enumerate}
\item We assume that cofibrations are closed under pullback.
\item We assume that cofibrations are closed under binary unions.
\item We are given connections on $\intv$ as defined in
  \cite{vdbergfrumin}.
\item We assume $\delta_0$ and $\delta_1$ are disjoint as subobjects
  of $\intv$.
\item We assume that $\delta_0$ and $\delta_1$ are cofibrations.
\item We assume that every map factors as a cofibration followed by a
  trivial fibration.
\item We assume that for every map $X$, the map $0 \to X$ is a
  cofibration. That is, every object is cofibrant.
\end{enumerate}

There are two main ways to satisfy the requirement that every map
factors as a cofibration followed by a trivial fibration. One way is
to assume that cofibrations form a dominance on $\catc$, which is the
case if the remaining Orton-Pitts axioms are assumed (this is the
approach taken by Van den Berg and Frumin). The other way is to assume
cofibrations are the left class in a cofibrantly generated wfs,
generated using a version of the small object argument. This could be
an approach using external transfinite colimits, such as Garner's
small object argument \cite{garnersmallobject}, but could also be an
internal version such as the one developed by the author in
\cite{swanwtypered}.

We require that path objects are, by definition, constructed using
exponentiation with the interval, as defined below.
\begin{definition}
  Given any object $X$, we define the \emph{path object on $X$}, to be
  the object $X^\intv$ (which we will also denote $PX$) together with
  the maps $r^X$, $p^X_0$ and $p^X_1$, where
  $r^X \colon X \to X^\intv$ is the constant map, and
  $p_0^X, p_1^X \colon X^\intv \to X^1 \cong X$ are given by
  composition with $\delta_0$ and $\delta_1$ respectively.
\end{definition}

The link between identity types in type theory and very good path objects
is one of the key ideas in homotopy type theory. See for example the
well known results of Gambino and Garner in \cite{gambinogarner} or
Awodey and Warren in \cite{awodeywarren}. In order for path types to
be used as identity types along these lines it is necessary for $r^X$
to be a trivial cofibration. We focus on the condition that $r^X$ is
just a cofibration, which of course follows from the assumption that
$r^X$ is a trivial cofibration\footnote{In fact with a little work
  one can show the converse also holds, but we don't need that here,
  since we are not constructing identity types but giving conditions
  that imply path objects are \emph{not} identity types.}.
\begin{definition}
  We say \emph{path types are identity types} if for every fibrant
  object $X$, the map $r^X$ is a cofibration.
\end{definition}

Recall that $P$ can be extended to a fibred functor over $\cod$ as
follows.

Given a map $f \colon X \to Y$, we define $P(f)$ to be given by the
pullback below, where the bottom map $Y \to P(Y)$ corresponds to the
projection $Y \times \intv \to Y$ under the adjunction.
\begin{equation*}
  \xymatrix{ P(f) \pullbackcorner \ar[r] \ar[d] & P(X) \ar[d] \\
    Y \ar[r] & P(Y)}
\end{equation*}
When $f$ is clear from the context, we will also write $P(f)$ as
$P_Y(X)$. One can further define maps $r^f \colon X \to Y$,
$p_0^f \colon P_Y(X) \to X$ and $p_1^f \colon P_Y(X) \to X$ to produce
a factorisation of the diagonal map $\Delta \colon X \to X \times_Y X$
in the slice category $\catc/Y$.

Note that for $Y = 1$, $P_Y(X)$ is just $P X$.

For some of our results, including the main theorem, we will need a
notion of propositional truncation. For this, we use the definitions
below.
\begin{definition}
  We say $f \colon X \to Y$ is an \emph{hproposition} if it is a
  fibration and the map $P_Y(X) \to X \times_Y X$ is a trivial
  fibration.
\end{definition}

We state below what it means for $\catc$ to have propositional
truncation, although technically we will never require this to hold
for $\catc$ itself. Instead we will assume that small maps are closed
under propositional truncation, in a sense that we will define later
(definition \ref{def:smallproptrunc}).
\begin{definition}
  We say $\catc$ \emph{has propositional truncation} if every fibration
  $f \colon X \to Y$ factors as follows, where $g$ is an
  hproposition and $i$ has the left lifting property against every
  hproposition.
  \begin{equation*}
    \xymatrix{ X \ar[rr]^f \ar[dr]_i & & Y \\
      & \| X \| \ar[ur]_g &}
  \end{equation*}
\end{definition}

We note that the axioms suffice to check a few basic propositions.
\begin{proposition}
  Every trivial cofibration is a cofibration.
\end{proposition}

\begin{proof}
  It suffices to show that every generating trivial cofibration is a
  cofibration. That is, for every cofibration $m \colon A \to B$, and
  for $i = 0, 1$, $\delta_i \hat{\times} m$ is a cofibration. However,
  the pushout product $\delta_i \hat{\times} m$ can be viewed as a
  union of cofibrations into $B \times \intv$, and we assumed that
  cofibrations are closed under finite union.
\end{proof}

\begin{proposition}
  Every trivial fibration has a section.
\end{proposition}

\begin{proof}
  This easily follows from the assumption that every object is
  cofibrant.
\end{proof}

Recall from \cite[Section 3.1]{vdbergfrumin} that $\intv$ can be used
to define a notion of \emph{homotopy} and so also \emph{homotopy
  equivalence}, as well as the stronger notion of \emph{strong
  homotopy equivalence}.

\begin{proposition}
  \label{prop:trivfibshe}
  A fibration $f \colon X \to Y$ between fibrant objects $X$ and $Y$
  is a trivial fibration if and only if it is a strong homotopy
  equivalence.
\end{proposition}

\begin{proof}
  See \cite[Proposition 4.1]{vdbergfrumin}
\end{proof}

\begin{proposition}
  \label{prop:mappingpathfact}
  Every map between fibrant objects factors as a homotopy equivalence
  followed by a fibration.
\end{proposition}

\begin{proof}
  See \cite[Proposition 4.3]{vdbergfrumin}.
\end{proof}

\begin{proposition}
  Dependent products preserve fibrations.
\end{proposition}

\begin{proof}
  See \cite[Proposition 4.5]{vdbergfrumin}.
\end{proof}

\begin{proposition}
  \label{prop:embedinfibrant}
  Every object is a subobject of a fibrant object.
\end{proposition}

\begin{proof}
  Given an object $X$, we factorise the map $X \to 1$ as a cofibration
  followed by a trivial fibration, to get $X \to X' \to 1$, where the
  map $X \to X'$ is a cofibration and in particular a monomorphism and
  the map $X' \to 1$ is a trivial fibration, and so in particular a
  fibration.
\end{proof}

\begin{proposition}
  Hpropositions are closed under arbitrary pullbacks.
\end{proposition}

\begin{proof}
  This follows from the fact that path types, fibrations and trivial
  fibrations are closed under pullbacks.
\end{proof}

\begin{proposition}
  \label{prop:hpropwsectistrivfib}
  If $f \colon X \to Y$ is an hproposition and has a section then it
  is a trivial fibration.
\end{proposition}

\begin{proof}
  Let $s \colon Y \to X$ be a section of $f$ and let
  $t \colon X \times_Y X \to P_Y X$ be a section of the map
  $P_Y X \to X \times_Y X$. We exhibit $f$ as a retract of the trivial
  fibration $p_1^f \colon P_Y(X) \to X$, which shows it is also a
  trivial fibration.
  \begin{equation*}
    \xymatrix{ X \ar[r]^{\langle 1_X, s \circ f \rangle} \ar[d]_f & X \times_Y
      X \ar[r]^t & P_Y X \ar[r]^{p_0^f} \ar[d]^{p_1^f} & X \ar[d]^f \\
      Y \ar[rr]_s & & X \ar[r]_f & Y}
  \end{equation*}
\end{proof}

\begin{lemma}
  \label{lem:mergeterms}
  Suppose we are given two monomorphisms $m_0 \colon A_0 \to B$ and
  $m_1 \colon A_1 \to B$, with at least one of $m_0$ and $m_1$ a
  cofibration, and an hproposition $f \colon X \to B$ together with two
  maps $t_0$ and $t_1$ in the following commutative diagram.
  \begin{equation*}
    \begin{gathered}
      \xymatrix{ & X \ar[d]^f & \\
        A_0 \ar[r]_{m_0} \ar[ur]^{t_0} & B & A_1 \ar[l]^{m_1} \ar[ul]_{t_1}}
    \end{gathered}
  \end{equation*}
  Write the union of $m_0$ and $m_1$ as $m \colon A_0 \cup A_1 \to B$.

  Then there is a map $t \colon A_0 \cup A_1 \to X$ making the
  following diagram commute.
  \begin{equation*}
    \begin{gathered}
      \xymatrix{ & X \ar[d]^f \\
        A_0 \cup A_1 \ar[r]_{m} \ar[ur]^{t} & B}      
    \end{gathered}
  \end{equation*}
\end{lemma}

\begin{proof}
  Without loss of generality say that $m_0$ is a cofibration.
  
  First, observe that the result would be trivial if we knew that
  $t_0$ and $t_1$ agreed on $A_0 \cap A_1$. We therefore aim to
  produce a new map $t_1'$ ensuring that $t_0$ and $t_1'$ agree on
  $A_0 \cap A_1$.

  If we pullback $f$ along $m_1$, then the resulting map
  $m_1^\ast(f) \colon m_1^\ast(X) \to A_1$ is also an hproposition,
  and using $t_1$ we can show it has a section. We deduce by
  proposition \ref{prop:hpropwsectistrivfib} that it is a trivial
  fibration. Furthermore, observe that the inclusion
  $\iota_1 \colon A_0 \cap A_1 \rightarrow A_1$ is a pullback of
  $m_0$, and so a cofibration. We will define a lifting problem of
  $\iota_1$ against $m_1^\ast(f)$. Let
  $\bar{t_0} \colon A_0 \cap A_1 \to m_1^\ast(X)$ be the pullback of $t_0$
  along $m_1$, so that if $\pi_0 \colon m_1^\ast(X) \to X$ is one of
  the projections in the pullback then
  $\pi_0 \circ \bar{t_0} = t_0 \circ \iota_0$. Then let $j$ be a diagonal
  filler in the following diagram.
  \begin{equation*}
    \xymatrix{ A_0 \cap A_1 \ar[d]_{\iota_1} \ar[r]^{\bar{t_0}} &
      m_1^\ast(X) \ar[d]^{m_1^\ast(f)} \\
      A_1 \ar@{=}[r] \ar@{.>}[ur]|j & A_1}
  \end{equation*}
  We then define $t_1'$ to be $\pi_0 \circ j$. One can then verify
  that $f \circ t_1' = m_1$ and
  $t_1' \circ \iota_1 = t_0 \circ \iota_0$, and we can now easily
  define the required $t$ using the universal property of the union.
\end{proof}

\begin{proposition}
  \label{prop:discretedefs}
  Let $X$ be an object of $\catc$. The following are equivalent.
  \begin{enumerate}
  \item \label{reflisiso}The map $r^X \colon X \to X^\intv$ is an isomorphism.
  \item \label{reflisepi}The map $r^X \colon X \to X^\intv$ is a
    regular epimorphism.
  \item \label{allpathsconst}The statement ``every function
    $\intv \to X$ is constant'' holds in the internal language of
    $\catc$.
  \end{enumerate}
\end{proposition}

\begin{proof}
  Note that $r^X$ is a split monomorphism in any case (with retraction
  $X^{\delta_0}$). It follows that \ref{reflisiso} and \ref{reflisepi}
  are equivalent.

  Showing \ref{reflisepi} and \ref{allpathsconst} are equivalent is
  straightforward.
\end{proof}

\begin{definition}
  We say an object $X$ is \emph{discrete} if one of the equivalent
  conditions in proposition \ref{prop:discretedefs} holds.
\end{definition}

\begin{proposition}
  \label{prop:connecteddefs}
  Let $X$ be an inhabited object of $\catc$. The following are
  equivalent.
  \begin{enumerate}
  \item The constant function map $2 \to 2^X$ is an isomorphism.
  \item The constant function map $2 \to 2^X$ is an epimorphism.
  \item The statement ``all functions from $X$ to $2$ are constant''
    holds in the internal language.
  \item (When $\catc$ has a subobject classifier) the following
    statement holds in the internal language ``if $U$ and $V$ are
    disjoint subobjects of $X$ such that $X = U \cup V$, then either
    $X = U$ or $X = V$.''
  \end{enumerate}
\end{proposition}

\begin{definition}
  We say an inhabited object $X$ is \emph{connected} if one of the
  equivalent conditions in proposition \ref{prop:connecteddefs} holds.
\end{definition}

\begin{proposition}
  \label{prop:discreteproperties}
  \begin{enumerate}
  \item If $X$ has decidable equality and the interval is connected,
    then $X$ is discrete.
  \item If $X$ is discrete and $Y$ is any object, then $X^Y$ is
    discrete.
  \item If $X$ is discrete and $m \colon Z \rightarrowtail X$ is a
    subobject then $Z$ is discrete.
  \item \label{repintvconstdisc}
    If $\catc$ is a category of presheaves over a category
    $\smcat{C}$, the interval is representable and $\smcat{C}$ has
    finite products then every constant presheaf is discrete.
  \item \label{mapbtwndiscisfib}
    Every map between discrete objects is a fibration.
  \end{enumerate}
\end{proposition}

\begin{proof}
  \ref{repintvconstdisc} was already observed by Uemura in
  \cite[Proposition 4.7]{uemuracubasm}. The rest are straightforward.
\end{proof}

Finally we recall the following notions of $\neg \neg$-stability,
density and separation.
\begin{definition}
  Let $m \colon X \rightarrow Y$ be a monomorphism in a category
  $\catc$.  We say $m$ is \emph{$\neg \neg$-stable} if the following
  statement holds in the internal language of $\catc$.
  \begin{equation*}
    \forall y \in Y, \; \neg \neg (\exists x \in X, m(x) = y) \,\rightarrow\,
    (\exists x \in X, m(x) = y)
  \end{equation*}

  We say $m$ is \emph{$\neg \neg$-dense} if the following
  statement holds in the internal language.
  \begin{equation*}
    \forall y \in Y, \; \neg \neg (\exists x \in X, m(x) = y)
  \end{equation*}

  We say an object $X$ is \emph{$\neg \neg$-separated} if the diagonal
  map $X \to X \times X$ is $\neg \neg$-stable, or equivalently, if
  the following statement holds in the internal language.
  \begin{equation*}
    \forall x, y \in X, \; \neg \neg(x = y) \,\rightarrow\, x = y
  \end{equation*}

  Suppose $\catc$ is a category of internal presheaves over an
  internal category $\smcat{C}$ in a locally cartesian closed category
  $\cat{E}$. We say a monomorphism $m$ in $\catc$ is \emph{pointwise
    $\neg \neg$-stable} if the underlying map in
  $\cat{E}/\operatorname{Ob}(\smcat{C})$ is $\neg \neg$-stable, or
  equivalently if the following statement holds in the internal
  language of $\cat{E}$: for every object $c$ of $\smcat{C}$, $m_c$ is
  $\neg \neg$-stable. We similarly define \emph{pointwise
    $\neg \neg$-dense} and \emph{pointwise $\neg \neg$-separated}.
\end{definition}

\section{Cofibrations when Path Types are Identity Types}
\label{sec:cofibr-when-path}

In this section we will use the assumption that path types are
identity types to show that certain maps are cofibrations.

\begin{lemma}
  \label{lem:constiscof}
  Suppose that $\intv$ is connected. Suppose that $\catc$ possesses a
  natural number object $\nat$. Suppose that path types are identity
  types. Then the map $1 \to 2^\nat$ given by $\lambda x.\lambda n.0$
  is a cofibration.
\end{lemma}

\begin{proof}
  First, note that $\intv$ is not necessarily fibrant, but we can
  embed it in a fibrant object $\fintv$ using proposition
  \ref{prop:embedinfibrant}. Write $\iota$ for the inclusion $\intv
  \to \fintv$.

  Since $\intv$ is connected, $\nat$ is fibrant. Since dependent
  products preserve fibrations, $\fintv^\nat$ is also fibrant. Hence
  by the assumption that path types are identity types,
  $r^{\fintv^\nat}$ is a cofibration. We will show that the map $1 \to
  2^\nat$ is a cofibration by exhibiting it as a pullback of
  $r^{\fintv^\nat}$.

  We define the map $e \colon 2^\nat \to (\fintv^\nat)^\intv$ as
  below.
  \begin{equation*}
    e(\alpha)(i)(n) =
    \begin{cases}
      \iota(0) & \alpha(n) = 0 \\
      \iota(i) & \alpha(n) = 1
    \end{cases}
  \end{equation*}

  To check that $\lambda x.\lambda n.0$ is a pullback of
  $r^{\fintv^\nat}$ along $e$, it suffices to show, in the internal
  logic of $\catc$, that for all
  $\alpha \in 2^\nat$, $\alpha = \lambda n.0$ if and only if
  $e(\alpha)$ lies in the image of $r^{\fintv^\nat}$, which is the
  case if and only if $e(\alpha)$ is constant as a function
  $\intv \to \fintv^\nat$.

  It is easy to check that $e(\lambda n.0)$ is constant. Hence we just
  show the converse, that if $e(\alpha)$ is constant then for all $n
  \in \nat$, $\alpha(n) = 0$.

  So suppose that $\alpha \in 2^\nat$ and $e(\alpha)$ is constant. Let
  $n$ be an element of $\nat$. We know that $\alpha(n) = 0$, or
  $\alpha(n) = 1$, so to show that $\alpha(n) = 0$ it suffices to show
  that $\alpha(n) \neq 1$. Suppose that $\alpha(n) = 1$. Since
  $e(\alpha)$ is constant, for all $i, i' \in \intv$ we have
  $e(\alpha)(i)(n) = e(\alpha)(i')(n)$, and so $\iota(i) =
  \iota(i')$. In particular, applying this to the endpoints $0$ and
  $1$, we have $\iota(0) = \iota(1)$ and so $0 = 1$ since $\iota$ is
  monic. But then we get a contradiction by the disjointness of the
  endpoints. We have now shown $\alpha(n) \neq 1$, and so
  $\alpha(n) = 0$. This applies for arbitrary $n$, and so
  $\alpha = \lambda n.0$.
\end{proof}

We next show that if we have exact quotients then we can in fact show
that \emph{all} monomorphisms are cofibrations. This doesn't apply in
presheaf assemblies, but does work for the usual definition of
presheaves and in fact for any $\Pi$-pretopos.

\begin{theorem}
  \label{thm:allmonoscofs}
  Suppose that $\catc$ is a $\Pi$-pretopos (i.e. $\catc$ has exact
  quotients) and path types are identity types. Then any monomorphism
  is a cofibration.
\end{theorem}

\begin{proof}
  Let $m \colon A \to B$ be any monomorphism. We need to
  show it is a cofibration.

  Working in the internal logic of $\catc$, we define an equivalence
  relation on $\intv \times B$. Given $(i, b)$ and $(i', b')$ in
  $\intv \times B$, we set $(i, b) \sim (i', b')$ if $b = b'$ and
  either $i = i'$ or $m^{-1}(\{b\})$ is inhabited. It is
  straightforward to verify that this is an equivalence relation.

  By proposition \ref{prop:embedinfibrant} there exists a fibrant
  object $X$ such that $\intv \times B / {\sim}$ is a subobject of
  $X$. Say $\iota \colon \intv \times B / {\sim} \to X$ is the
  subobject inclusion.

  We define $e \colon B \to X^\intv$ to be the map that sends each $b
  \in B$ to $\lambda i.\iota([(i, b)])$.

  Now, still reasoning internally in $\catc$, we show that for each $b
  \in B$, $m^{-1}(\{b\})$ is inhabited if and only if $e(b)$ lies in
  the image of $r^X$.

  Suppose first that $m^{-1}(\{b\})$ is inhabited. Then for any
  $i, i' \in \intv$, we have $[(i, b)] = [(i', b)]$. Hence
  $e(b) = \lambda i.\iota([(i, b)])$ is a constant function, and so
  lies in the image of $r^X$.

  Conversely, suppose that $e(b) = \lambda i.\iota([(i, b)])$ lies in
  the image of $r^X$. Then it is a constant function, and in
  particular we have $\iota([(0, b)]) = \iota([(1, b)])$, where $0$
  and $1$ are the images of $\delta_0$ and $\delta_1$
  respectively. Since $\iota$ is monic, we deduce $[(0, b)] = [(1,
  b)]$. Since quotients are exact, we now have $(0, b) \sim (1, b)$,
  and so either $0 = 1$, or $m^{-1}(\{b\})$ is inhabited. But we
  assumed the endpoints are disjoint and so we do have that
  $m^{-1}(\{b\})$ is inhabited as required.

  Since $r^X$ is monic, we can now deduce that $e$ fits into a
  pullback diagram as below.
  \begin{equation*}
    \xymatrix{ A \ar[d]_m \ar[r] \pullbackcorner & X \ar[d]^{r^X} \\
      B \ar[r]_e & X^\intv}
  \end{equation*}
  This witnesses $m$ as a pullback of the cofibration $r^X$, and so is
  itself a cofibration, as required.
\end{proof}

\section{Cofibrations and Univalent Universes}
\label{sec:cofibr-univ}

We specialise to the case where $\catc$ is a category of internal
presheaves. That is, we fix a locally cartesian closed category
$\cat{E}$ with all finite colimits and an
internal category $\smcat{C}$ in $\cat{E}$. We then take $\catc$ to be
the category of internal presheaves over $\smcat{C}$. We will follow
the convention that $\catc$ is non trivial, in the sense that
$\smcat{C}$ has at least one object.

\begin{definition}
  Suppose we are given a map $\elu \colon \tilde{U} \to U$. We say a
  map $f \colon X \to Y$ is \emph{$U$-small} if it is a pullback of
  $\elu$ along a map $g \colon Y \to U$. We will refer to such a $g$
  as a \emph{classifying map} for $f$.
\end{definition}

\begin{definition}
  We say $\elu \colon \tilde{U} \to U$ is a \emph{universe} if the
  following hold.
  \begin{enumerate}
  \item Every isomorphism is $U$-small.
  \item $U$-Small maps are closed under composition.
  \item $U$-Small maps are closed under dependent product.
  \item $U$-Small maps are closed under pairwise coproduct.
  \item $U$-Small maps are closed under mapping path spaces.
  \end{enumerate}

  We say a universe $\elu \colon \tilde{U} \to U$ is a
  \emph{homotopical universe} if in addition to the above, we have the
  following.
  \begin{enumerate}
  \item $\elu$ is a fibration (or equivalently every $U$-small map is
    a fibration).
  \item $U$ is fibrant.
  \end{enumerate}
\end{definition}

\begin{definition}
  \label{def:smallproptrunc}
  We say a homotopical universe $U$ is \emph{closed under propositional
    truncation} if every $U$-small fibration $f$ factors as a map with
  the left lifting property against all hpropositions followed by a
  $U$-small hproposition.
\end{definition}

We will often view $\elu \colon \tilde{U} \to U$ as a family of types
$\elu(x)$ indexed by the elements $x$ of $U$.

Note that using path objects and dependent products, we can translate
one of the definitions of equivalence from type theory (as appear for
instance in \cite[Chapter 4]{hottbook}) into the formulation we are
using here. We write $\hequiv(X, Y)$ for the object of equivalences
from $X$ to $Y$. Observe that for any of the usual definitions it is
straightforward to show that every isomorphism is an equivalence. We
use $\hequiv(X, Y)$ to define univalence as follows.
\begin{definition}
  We say a homotopical universe $\elu \colon \tilde{U} \to U$ is
  \emph{univalent} if the first projection
  $\Sigma_{X : U} \Sigma_{Y : U} \hequiv(X, Y) \to U$ is a trivial
  fibration.
\end{definition}

\begin{remark}
  In any case $\Sigma_{X : U} \Sigma_{Y : U} \hequiv(X, Y) \to U$ is a
  fibration by our other conditions. Hence by proposition
  \ref{prop:trivfibshe} it is a trivial fibration if and only if the
  statement that it is contractible holds in the model. This is
  equivalent to the univalence axiom holding in the model.
\end{remark}

\begin{definition}
  We fix a map $u \colon 1 \to U$ such that the following is a
  pullback, and refer to it as the \emph{unit type}.
  \begin{equation*}
    \xymatrix{ 1 \ar[r] \ar[d] \pullbackcorner & \tilde{U} \ar[d] \\
      1 \ar[r]_u & U}
  \end{equation*}

  When $\catc$ is a category of (possibly internal) presheaves over a
  category $\smcat{C}$ we use the following notation.
  Since $u$ is a global section of the presheaf $U$, we can think of
  it as a choice of elements $u(c)$ for each object $c$ in the
  category $\smcat{C}$.
\end{definition}

We clearly have the following proposition.
\begin{proposition}
  Suppose that $f \colon X \to Y$ is a pullback of $\elu \colon
  \tilde{U} \to U$ along a map of the form $u \circ !_Y$, where $!_Y$
  is the unique map $Y \to 1$. Then $f$ is an isomorphism.
\end{proposition}

We will also need the following observations.
\begin{proposition}
  \label{prop:coftofib}
  Let $U$ be a homotopical universe. Suppose that $m \colon A \to B$
  is a cofibration, $A$ and $B$ are both $U$-small objects and $B$ is
  discrete. Then $m$ is a small fibration and furthermore an
  hproposition.
\end{proposition}
\begin{proof}
  We first note that since $B$ is discrete, the subobject $A$ must be
  too.  Since $m$ is a map between discrete objects it is a fibration
  by proposition \ref{prop:discreteproperties}. However, we still need
  to show that it is a small fibration. We first replace $A$ with the
  mapping path space, which we view as a small fibration $A' \to
  B$. Explicitly, we can define $A'$ as a type internally in type
  theory with the following definition (where the equality is
  implemented using path types).
  \begin{equation*}
    A'(b) := \Sigma_{a : A} m(a) = b
  \end{equation*}
  This gives a well defined small fibration since $B$ and $A$ are small
  types, and universes are closed under path types. We now note that
  since $B$ is discrete, we in fact have an isomorphism $A' \cong A$, and so
  $m$ is a small fibration as required.
\end{proof}

\begin{lemma}
  \label{lem:unionhunionequiv}
  Let $U$ be a homotopical universe closed under propositional
  truncation. Suppose that $m_i \colon A_i \to B$ are cofibrations,
  $A_i$ and $B$ are $U$-small objects and $B$ is discrete (but
  $A_0 \cup A_1$ is not necessarily small). By proposition
  \ref{prop:coftofib} a $U$-small propositional truncation
  $\| A_0 + A_1 \|$ exists. Then we have maps
  $A_0 \cup A_1 \to \| A_0 + A_1 \|$ and
  $\| A_0 + A_1 \| \to A_0 \cup A_1$ forming commutative triangles in
  the following diagram.
  \begin{equation*}
    \xymatrix{ A_0 \cup A_1 \ar@/^/[rr] \ar[dr] & & \| A_0 + A_1 \|
      \ar@/^/[ll] \ar[dl] \\
      & B & }
  \end{equation*}
\end{lemma}
\begin{proof}
  We first construct the map $A_0 \cup A_1 \to \| A_0 + A_1 \|$.
  
  We clearly have maps $t_i$ for $i = 0,1$ in the following commutative
  diagram.
  \begin{equation*}
    \xymatrix{ & \| A_0 + A_1 \| \ar[d] \\
      A_i \ar[ur]^{t_i} \ar[r]_{m_i} & B
    }
  \end{equation*}
  Since $\| A_0 + A_1 \|$ is an hproposition and each $m_i$ is a
  cofibration we can apply lemma \ref{lem:mergeterms} to get the
  required map $A_0 \cup A_1 \to \| A_0 + A_1 \|$.

  We next construct the map $\| A_0 + A_1 \| \to A_0 \cup A_1$.
  Since
  $m \colon A_0 \cup A_1 \to B$ is a map between discrete objects, it
  is a fibration, albeit not necessarily small. Since it is a
  monomorphism and a fibration, it is an hproposition.
  Hence we can obtain
  the required map as a filler in the following lifting problem.
  \begin{equation*}
    \xymatrix{ A_0 + A_1 \ar[r] \ar[d]_{|-|} & A_0 \cup A_1 \ar[d]^m \\
      \| A_0 + A_1 \| \ar[r] \ar@{.>}[ur] & B}
  \end{equation*}
\end{proof}

We will now see the first key lemma of this section. We will later give
a more general lemma, but this one is simpler and therefore easier to
understand, and is already useful in presheaf assemblies where all
objects are pointwise $\neg \neg$-separated.
\begin{lemma}
  \label{lem:vsephs}
  Let $\catc$ be a category of internal presheaves in a locally
  cartesian closed category with finite colimits and disjoint
  coproducts. Suppose we are given a univalent universe
  $\elu \colon \tilde{U} \to U$ and two maps $m_0 \colon A_0 \to B$
  and $m_1 \colon A_1 \to B$ satisfying the following conditions.
  \begin{enumerate}
  \item $U$ is pointwise $\neg \neg$-separated.
  \item $U$ is closed under propositional truncation.
  \item Both $m_0$ and $m_1$ are cofibrations.
  \item $B$ is discrete.
  \item $A_0$, $A_1$ and $B$ are $U$-small (but note that $A_0
    \cup A_1$ does \emph{not} need to be $U$-small).
  \end{enumerate}
  Write $m$ for the union $A_0 \cup A_1 \rightarrow B$.

  Then $m$ is pointwise $\neg \neg$-stable.
\end{lemma}

\begin{proof}
  We first note that since $B$ is discrete, each $m_i$ is a small
  fibration by proposition \ref{prop:coftofib}.

  We avoided assuming that $A_0 \cup A_1$ is small. We note however,
  that the ``homotopy union'' of $A_0$ and $A_1$ is necessarily small,
  since $U$ is closed under coproducts and propositional truncation.
  Explicitly, we define another small fibration
  $f \colon C \to B$ using the definition below.
  \begin{equation*}
    C(b) := \| A_0(b) + A_1(b) \|
  \end{equation*}
  Let $\gamma \colon B \to U$ be a classifying map for $f$.

  Since cofibrations are closed under unions, $m$ is a
  cofibration. Since $U$ is univalent the map
  $\pi_X \colon \Sigma_{X : U} \Sigma_{Y : U} \hequiv(X, Y) \to U$ is
  a trivial fibration. We will aim to define a lifting problem of $m$
  against $\pi_X$, as illustrated below.
  \begin{equation}
    \label{eq:vsephslp}
    \begin{gathered}
      \xymatrix{ A_0 \cup A_1 \ar[r]^(0.35)\alpha \ar[d]_m & \Sigma_{X : U}
        \Sigma_{Y : U} \hequiv(X, Y) \ar[d]^{\pi_X} \\
        B \ar[r]_\gamma \ar@{.>}[ur]|j & U}      
    \end{gathered}
  \end{equation}
  We take the bottom map $B \to U$ to be $\gamma$,
  which we recall was a classifying map for the small fibration
  $f \colon C \to B$.

  The next step is to construct the top map $\alpha$ of the lifting
  problem, which needs to map from $A_0 \cup A_1$ to
  $\Sigma_{X : U} \Sigma_{Y : U} \hequiv(X, Y)$. Note that this
  amounts to constructing maps $\xi, \zeta \colon A_0 \cup A_1 \to U$
  together with an equivalence $e$ over $A_0 \cup A_1$ between
  $\xi^\ast(\tilde{U})$ and $\zeta^\ast(\tilde{U})$. First note that
  in order for the lifting problem to be a commutative square, we are
  forced to take $\xi$ to be $\gamma \circ m$.

  The key to the proof is that we define $\zeta$ to be
  $u \circ !_{A_0 \cup A_1}$. Informally, the $Y$ component of the map
  from $A_0 \cup A_1$ to $\Sigma_{X : U} \Sigma_{Y : U} \hequiv(X, Y)$
  is constantly equal to the unit type.

  It still remains to construct the equivalence $e$. Since the
  definition of equivalence does not require defining small types we
  no longer need to work ``inside $U$.'' Therefore, as we stated
  above, it suffices to construct an equivalence in $\catc$ over
  $A_0 \cup A_1$ between $m^\ast(f)$ and the identity on
  $A_0 \cup A_1$. Note that it suffices to show that the map
  $m^\ast(f) \colon m^\ast(C) \to A_0 \cup A_1$ is a
  strong homotopy equivalence. Hence
  by proposition \ref{prop:trivfibshe} it suffices to show it is a
  trivial fibration. Recall that we constructed $f \colon C \to B$
  by interpreting
  the type $\| A_0(b) + A_1(b) \|$, which is an hproposition. Hence the
  pullback $m^\ast(f)$ is also an hproposition. Therefore to show it
  is a trivial fibration, it suffices by proposition
  \ref{prop:hpropwsectistrivfib} to show it has a section, which
  easily follows from lemma \ref{lem:unionhunionequiv}.
  
  So we do have a well defined map
  $\alpha \colon A_0 \cup A_1 \to \Sigma_{X : U} \Sigma_{Y : U}
  \hequiv(X, Y)$ such that $\pi_X \circ \alpha = \gamma$ and
  $\pi_Y \circ \alpha = u \circ !_{A_0 \cup A_1}$. Let $j$ be a
  diagonal filler as in \eqref{eq:vsephslp}.

  We now use all of this to show $m$ is locally $\neg \neg$-stable. We
  recall that we are working in a category of presheaves over a
  category $\cat{E}$ and switch to the internal logic of
  $\cat{E}$. Let $c$ be an object of $\smcat{C}$ and let $b \in
  B(c)$. Suppose that $m_c^{-1}(b)$ is not empty. We need to show that
  it is inhabited.

  Note that if $m_c^{-1}(b)$ was inhabited, then the upper triangle in
  the lifting diagram would imply that $\pi_Y(j_c(b)) = u(c)$. We can
  therefore deduce that $\pi_Y(j_c(b))$ is not not equal to
  $u(c)$. However, we can now apply the fact that $U$ is pointwise
  $\neg \neg$-separated to show that in fact $\pi_Y(j_c(b))$ is equal
  to $u(c)$. Furthermore, for all $\sigma \colon c' \to c$ in
  $\smcat{C}$, we have that $m_{c'}^{-1}(B(\sigma)(b))$ is not empty,
  and so we similarly can show that $\pi_Y(j_{c'}(B(\sigma)(b)))$ is
  equal to $u_{c'}(B(\sigma)(b))$. Therefore, if
  $\bar{b} \colon \yoneda c \to B$ is the map corresponding to $b$
  under Yoneda, then the composition $\pi_Y \circ j \circ \bar{b}$
  factors through the unit type $u \colon 1 \to U$. Hence the pullback
  of $Y$ along $\bar{b}$ is an isomorphism. Furthermore, we can
  pullback the equivalence to obtain an equivalence between
  $\bar{b}^\ast(Y)$ and $\bar{b}^\ast\| A_0 + A_1 \|$. We deduce that
  $\bar{b}^\ast\| A_0 + A_1 \|$ has a section. This gives us an
  element of $\| A_0 + A_1 \|(c)$ in the fibre of $b$.

  Finally, applying the map $\| A_0 + A_1 \| \to A_0 \cup A_1$ from
  lemma \ref{lem:unionhunionequiv} gives us an element of
  $A_0 \cup A_1 (c)$ in the fibre of $b$ as required.
\end{proof}

We will now aim towards another, more general result, which allows us
to replace the assumption that $U$ is $\neg \neg$-separated with a
much weaker (but more complicated) requirement. We will further assume
that the universe contains ``contractibility representations'' in a
sense that we will define below.

\begin{definition}
  Let $V$ be a universe in $\cat{E}$. A \emph{weakly $\neg
    \neg$-stable unit} is a unit $u \colon 1 \to V$ such that the
  following holds in the internal logic of $\cat{E}$.
  \begin{enumerate}
  \item $\elv(u)$ has exactly one element.
  \item For all $x \in V$, if $\neg \neg (x = u)$, then $\elv(x)$ has at
    most one element.
  \end{enumerate}
\end{definition}

A key idea is that although the definition of weakly
$\neg \neg$-stable unit still sounds a little strong when working in
intuitionistic logic, it does hold in constructive set theory, using
the axiom of extensionality. An earlier version of this idea is
mentioned by Orton and Pitts in \cite[Remark
8.7]{pittsortoncubtopos}.
\begin{lemma}
  \label{lem:czfwkstbunit}
  Work over $\czf + \inacc$, take $\cat{E}$ to be the category of
  sets, and $V$ to be an inaccessible set. Then $V$ has a weakly $\neg
  \neg$-stable unit.
\end{lemma}

\begin{proof}
  We take $z$ to be the small set $\{\emptyset\}$ (which is
  the usual implementation of the terminal object in $\set$
  anyway). Suppose that $x$ is a (small) set and that the double
  negation of $x = z$ holds. Now let $y$ be any element of
  $x$. Suppose that $y$ contains an element $w$. Then $y \neq
  \emptyset$, and so $y \notin z$. Hence $x \neq z$ by extensionality,
  contradicting the double negation of $x = z$. But we have now shown
  that every element $y$ of $x$ is empty, and so $x$ is a subset of
  $\{\emptyset\}$. We can now deduce that $x$ has at most one element,
  as required.
\end{proof}

We note furthermore that there is another example of a weakly
$\neg \neg$-stable unit in the effective topos, or more generally any
realizability topos. In \cite[Section
3]{streicherunivtop}, Streicher observed that one can construct
universes in realizability toposes using ideas developed by Awodey,
Butz, Simpson and Streicher in \cite{abss}, as follows.

Assuming the existence of an inaccessible ordinal $\kappa$ one can
obtain a set sized version of McCarty's model of $\mathbf{IZF}$ from
\cite{mccarty} by truncating the definition of $V(\mathcal{A})$ at
level $\kappa$. Streicher then makes this into an object
$\mc(\mathcal{A})$ of the realizability topos using the same
definition of equality as used in set theory to obtain a universe.

\begin{lemma}
  \label{lem:mccartyhasnegnegunit}
  For any pca $\mathcal{A}$, the universe $\mc(\mathcal{A})$ defined
  above possesses a weakly $\neg \neg$-stable unit.
\end{lemma}

\begin{proof}
  Since $V_\kappa(\mathcal{A})$ is a model of set theory, we can carry
  out exactly the same argument as in lemma \ref{lem:czfwkstbunit}
  internally in the model. Since the definition of equality in the
  topos is the same as in the set theoretic model, it follows that
  we do get a weakly $\neg \neg$-stable unit in the topos.
\end{proof}

As before, we will exploit the fact that we are working in a category
of presheaves.
\begin{definition}
  Let $U$ be a universe in $\catc$. We say a \emph{pointwise weakly
    $\neg \neg$-stable unit} is a map $u \colon 1 \to U$ with the
  following property. In the internal logic of $\cat{E}$ we have that
  for every $c \in \smcat{C}$ and every $x \in U(c)$, if $\neg \neg x
  = u(c)$ then $\elv(c, x)$ has at most one element.
\end{definition}

\begin{definition}
  If $\cat{E}$ has a universe $V$, recall that we can define the
  Hofmann-Streicher universe $V_\smcat{C}$ in $\catc$ as
  follows. Given an object $c$ of $\smcat{C}$, we take
  $(V_\smcat{C})_c$ to be the collection of ``small presheaves'' on
  the category $\int_c \yoneda c$, as defined by Hofmann and Streicher
  in \cite{hofmannstreicherliftuniv}. The action of morphisms is
  defined via composition. Defining this internally in $\cat{E}$ takes
  a little care, but this has been done by Uemura in \cite[Section
  4.1]{uemuracubasm}.
\end{definition}

\begin{definition}
  Let $V$ be a universe in $\cat{E}$. A \emph{homotopical
    Hofmann-Streicher universe on $V$} is a homotopical universe $U$
  together with a map $i \colon U \to V_\smcat{C}$ with the following
  property. Let $\chi \colon Y \to V_\smcat{C}$ be any map. Then
  $\chi$ factors through $i$ if and only if the pullback of $\elv$
  along $\chi$ is a fibration.
\end{definition}

\begin{remark}
  In \cite{uemuracubasm}, Uemura used techniques developed by Licata,
  Orton, Pitts and Spitters in \cite{lops} to extend a
  Hofmann-Streicher universe in cubical assemblies to a homotopical
  Hofmann-Streicher universe.
\end{remark}

\begin{lemma}
  \label{lem:negnegunittopointwise}
  Suppose that $V$ has a weakly $\neg \neg$-stable unit and $U$ is a
  homotopical Hofmann-Streicher universe on $V$. Then $U$ has a
  pointwise weakly $\neg \neg$-stable unit.
\end{lemma}

\begin{proof}
  Suppose that $u \colon 1 \to V$ is a weakly $\neg \neg$-stable
  unit in $V$. First note that $V_\smcat{C}$ has a pointwise weakly
  $\neg \neg$-stable unit $u' \colon 1 \to V_\smcat{C}$ defined as
  follows. In the internal logic of $\smcat{E}$ we need to define a
  map $1 \to V_\smcat{C}(c)$ for each object $c$ of $\smcat{C}$. An
  element of $V_\smcat{C}(c)$ consists first of a map $\Sigma_{d \in
    \smcat{C}} \hom(d, c) \to V$. We define each such map to be constantly
  equal to $u$. We take the action on morphisms to be the identity
  everywhere.
  
  Next, note that the identity on $1$ is an isomorphism and so a
  fibration, and so $u'$ does factor through
  $i \colon U \to V_\smcat{C}$ to give a map $u'' \colon 1 \to U$.

  We finally need to check that $u''$ is pointwise weakly $\neg
  \neg$-stable. That is, we need to show, in the internal logic of
  $\cat{E}$, that for all $c \in \smcat{C}$ and every $x \in U(c)$, if
  $\neg \neg x = u''(c)$ then $\elu(x)$ has at most one element. First,
  note that if $\neg \neg x = u''(c)$, then also $\neg \neg i(x) =
  u'(c)$. It follows that for every $d \in \smcat{C}$, and every
  $f \colon d \to c$, we have $\neg \neg i(x)(d, f) = u$. In
  particular we have $\neg \neg i(x)(c, 1_c) = u$. Hence $\elv(i(x)(c,
  1_c))$ has at most one element. But this is precisely the definition
  of $\elv(i(x))(c)$.
\end{proof}

\begin{definition}
  Let $f \colon X \to Y$ be a small hproposition with respect to some
  universe $U$. We say a \emph{contractibility representation} for
  $f$ is a small fibration $g \colon Z \to Y$ that has a section $z_0
  \colon Y \to Z$, and
  such that there is a homotopy equivalence $e$ in the following
  diagram, where $\iscontr_{z_0} (Z)$ is the result of interpreting
  $\Pi_{z : Z} \, z = z_0$ using path types and dependent products in the
  usual way.
  \begin{equation*}
    \xymatrix{ X \ar[rr]^e \ar[dr]_f & & \iscontr_{z_0}(Z) \ar[dl] \\
      & Y & }
  \end{equation*}
\end{definition}

\begin{remark}
  One can construct contractibility representations working
  internally in homotopy type theory under reasonable conditions about
  the existence of higher inductive types. For example, given an
  hproposition $X$, one can show using univalence that if the
  suspension $\operatorname{Susp}(X)$ exists then it is a
  contractibility representation of $X$. Alternatively one can also
  use set quotients together with univalence.
\end{remark}

\begin{lemma}
  \label{lem:negnegunittocof}
  Suppose that $\catc$ is a category of presheaves over an
  internal category $\smcat{C}$ in $\cat{E}$ and that all of the
  following hold.
  \begin{enumerate}
  \item $\catc$ has a univalent universe $U$.
  \item $U$ has a pointwise weakly $\neg \neg$-stable unit.
  \item $U$ is closed under propositional truncation.
  \item $U$ has contractibility representations for all small
    hpropositions.
  \end{enumerate}

  Let $A_0$, $A_1$ and $B$ be small and discrete objects of
  $\catc$. If $m_i \colon A_i \to B$ are cofibrations for $i = 0, 1$
  then the union $m \colon A_0 \cup A_1 \to B$ is pointwise
  $\neg \neg$-stable.
\end{lemma}

\begin{proof}
  We start by following the same proof as for lemma
  \ref{lem:vsephs}. We recall that this allows us to define the
  small hproposition $f \colon C \to B$ defined as below.
  \begin{equation*}
    C(b) := \| A_0(b) + A_1(b) \|
  \end{equation*}

  Next, let $g \colon D \to B$ be a contractibility representation
  for $f$, and let $\beta \colon B \to U$ be a classifying map for
  $g$.

  As before, we next define a map
  $\alpha \colon A \to \Sigma_{X : U} \Sigma_{Y : U} \hequiv(X,
  Y)$ that we will use along with $\beta$ to define a lifting problem.

  Given $a \in A$, we need to define small types $X(a)$ and $Y(a)$, together
  with an equivalence $e$ between them. As before, we are forced to
  take $X(a)$ to be $\beta(m(a))$ in order for the square to commute. We
  take $Y(a)$ to be the pointwise weakly $\neg \neg$-stable unit.

  We now need to construct an equivalence between $m^\ast (D)$ and the
  unit type over $A$. First, following the proof of lemma
  \ref{lem:vsephs}, we note that we can use lemma \ref{lem:unionhunionequiv}
  to construct a section of $m^\ast(C)$. It follows that we can
  construct a section of $\iscontr(m^\ast(D))$, and so
  we obtain the required equivalence from the observation that any two
  contractible types are equivalent.

  Now since $m$ is a cofibration by assumption, we have a
  diagonal filler $j$ in the diagram below.
  \begin{equation}
    \label{eq:3}
    \begin{gathered}
      \xymatrix{ A \ar[r] \ar[d]_{m} & \Sigma_{X :
          U} \Sigma_{Y : U}
        \hequiv(X, Y) \ar[d] \\
        B \ar[r]_{\beta} \ar@{.>}[ur]|j & U}
    \end{gathered}
  \end{equation}

  We take $y \colon B \to U$ to be the composition of $j$ with the
  projection to the $Y$ component.  We will write $Y$ for the pullback
  of $\elu$ along $y$.

  We now recall that we are working in a category of internal
  presheaves and switch to the internal logic of $\cat{E}$. Let $c$ be
  an object of $\smcat{C}$ and let $b \in B(c)$. We will deduce that
  $m_c^{-1}(\{b\})$ is inhabited from its double negation. As before,
  note that if $m_c^{-1}(\{b\})$ is inhabited, then the upper triangle
  of \eqref{eq:3} implies that $y_c(b)$ is equal to $u$ for all
  $c \in \smcat{C}$.  Hence, if $m^{-1}(\{b\})$ is not not inhabited,
  then $y_c(b)$ is not not equal to $u(c)$. Since $u(c)$ is weakly
  $\neg \neg$-stable, we deduce that if $m_c^{-1}(\{b\})$ is not not
  inhabited then $Y(c, b)$ has at most one element for every
  $c \in \smcat{C}$. Hence, by the same argument as in lemma
  \ref{lem:vsephs}
  the pullback of $Y$ along the map
  $\bar{b} \colon \yoneda c \to B$ is an hproposition. Since the
  pullback of $D$ along $\bar{b}$ is equivalent to the
  pullback of $Y$ over $\yoneda c$, it is also an hproposition. We can
  now use the definition of contractibility representation to show that
  $\bar{b}^\ast(C)$ has a section.

  Finally, as in the proof of lemma \ref{lem:vsephs}, using lemma
  \ref{lem:unionhunionequiv} and the discreteness of $B$ we can
  deduce that $m_c^{-1}(\{b\})$ is inhabited.
\end{proof}

\section{The Counterexamples}

We now give the counterexamples. We first show that it is impossible
to take identity types to be path types in certain models of univalent
type theory in presheaf assemblies. We assume that the reader is
familiar with standard definitions and results in realizability. See
e.g. \cite{vanoosten} for a good introduction. Recall that the
\emph{lesser limited principle of omniscience} is defined as follows.

\begin{definition}
  The \emph{lesser limited principle of omniscience} ($\llpo$) states
  that if $\alpha \colon \nat \to 2$ is a binary sequence such that
  $\alpha(n) = 1$ for at most one $n$, then either $\alpha(2n) = 0$
  for all $n$, or $\alpha(2n + 1) = 0$ for all $n$.
\end{definition}

\begin{lemma}
  \label{lem:pathareidtollpo}
  Suppose the following.
  \begin{enumerate}
  \item $\cat{E}$ is a locally cartesian closed category with finite
    colimits and disjoint coproducts.
  \item $\catc$ is a category of internal presheaves in $\cat{E}$.
  \item $\catc$ possesses a class of cofibrations and an interval
    satisfying our general conditions.
  \item The interval object $\intv$ in $\catc$ is connected.
  \item $\catc$ has a univalent universe $U$, satisfying the following
    \begin{enumerate}
    \item $\nat$ is $U$-small
    \item $U$ is closed under propositional truncation.
    \item $U$ is pointwise $\neg \neg$-separated.    
    \end{enumerate}
  \item Path types are identity types.
  \end{enumerate}

  Then $\llpo$ holds in the internal logic of $\cat{E}$.
\end{lemma}

\begin{proof}
  We aim to apply lemmas \ref{lem:constiscof} and \ref{lem:vsephs}.
  
  We define $B$ to consist of those $\alpha(n)$ in $2^\nat$ such that
  $\alpha(n) = 1$ at most once. Note that $\Delta$ preserves
  exponentials, limits, colimits and the natural number object. Hence
  $\Delta B$ is also a subobject of $2^\nat$ in $\catc$ and so
  discrete.

  We define two subobjects $A_0$ and $A_1$ of $B$ as below.
  \begin{align*}
    A_0 &:= \{ \alpha \in B \;|\; \forall n \in \nat,\, \alpha(2n) = 0 \}
           \\
    A_1 &:= \{ \alpha \in B \;|\; \forall n \in \nat,\, \alpha(2n + 1) = 0 \}
  \end{align*}
  Observe that $\Delta(A_0)$ and $\Delta(A_1)$ can both be written as
  pullbacks of the constant map $1 \to 2^\nat$, which is a cofibration
  by lemma \ref{lem:constiscof}, and so the inclusions
  are both cofibrations.

  We need to check that $\Delta(B)$, $\Delta(A_0)$ and $\Delta(A_1)$
  are small. Since small maps are closed under dependent products,
  composition and path types, we can implement the types below as
  small fibrant objects, where equality is interpreted using path
  types.
  \begin{align*}
    B' &:= \Sigma_{\alpha : 2^\nat} \, \Pi_{n, m : \nat}\,
         \alpha(n) = 1 \,\times\, \alpha(m) = 1 \;\to\;
         n = m \\
    A_0' &:= \Sigma_{\alpha : 2^\nat} \, \Pi_{n : \nat} \, \alpha(2n)
           = 0 \\
    A_1' &:= \Sigma_{\alpha : 2^\nat} \, \Pi_{n : \nat} \, \alpha(2n +
           1) = 0
  \end{align*}
  Since all objects involved are discrete, the same definition applies
  whether we interpret these types ``extensionally'' (i.e. using the
  internal language of $\catc$ in the usual way) or ``intensionally''
  (i.e. using the homotopical structure, and in particular path types
  for the equalities). Furthermore, as we observed
  earlier $\Delta$ preserves limits, exponentials and the natural
  number object, and so we see that $\Delta(B)$, $\Delta(A_0)$ and
  $\Delta(A_1)$ are respectively isomorphic to $B'$, $A_0'$ and
  $A_1'$, and therefore small.
  
  We define $A$ to be the union of $A_0$ and $A_1$, with $m$ the
  inclusion $A \hookrightarrow B$. Since $\Delta$ preserves unions and
  cofibrations are closed under unions, $\Delta(m)$ is a
  cofibration.

  We can therefore apply lemma \ref{lem:vsephs} to show that $m$ is
  $\neg \neg$-stable. However one can check that $A$ is
  $\neg \neg$-dense in $B$, and that $\llpo$ precisely states that
  every element of $B$ belongs to $A$. It follows that $\llpo$ holds
  in $\cat{E}$.
\end{proof}

\begin{theorem}
  \label{thm:klassembliesnopathid}
  Let $\catc$ be a category of presheaf assemblies over either of the
  pca's $\klone$ or $\kltwo$. Assume the axiom of excluded middle in
  the meta theory. Then it is impossible to satisfy all of
  the following conditions.

  \begin{enumerate}
  \item $\catc$ possesses a class of cofibrations and an interval
    satisfying our general conditions.
  \item There is a univalent universe containing $\nat$ and closed
    under propositional truncation.
  \item The interval object $\intv$ in $\catc$ is connected.
  \item Path types are identity types.
  \end{enumerate}
\end{theorem}

\begin{proof}
  It is well known that in categories of assemblies every object is
  $\neg \neg$-separated, as long as we assume excluded middle in the
  meta theory (see e.g. \cite[Section 3.1]{vanoosten}).
  Hence any object in a category of internal presheaves over
  assemblies is pointwise $\neg \neg$-separated. In particular this
  applies to any universe.
  It is also well known that $\llpo$ fails in assemblies over
  $\klone$ and over $\kltwo$. For example, Richman proved in
  \cite[Theorem 5]{richmanllpon} that in the presence of countable
  choice a weak form of $\llpo$ is not consistent with Church's
  thesis. However, both Church's thesis and countable choice hold in
  the effective topos. A similar argument applies in function
  realizability (i.e. realizability over $\kltwo$); see
  e.g. the proof of \cite[Corollary 7.23]{rathjenswanlif}.

  We apply lemma \ref{lem:pathareidtollpo}.
\end{proof}

We now turn to examples based on exact categories, specifically
ordinary presheaves and internal presheaves in realizability toposes.

\begin{theorem}
  Let $\inacc$ be the axiom that every set is an element of an
  inaccessible set (where we define inaccessible to include closure
  under subsets). We work over $\izf + \inacc$.

  Suppose that the following hold.
  \begin{enumerate}
  \item We are given a small category $\smcat{C}$ with finite products.
  \item The category of presheaves $\preshf{\smcat{C}}$ has a class of
    cofibrations and an interval satisfying our general conditions.
  \item The interval is connected.
  \item $\catc$ possesses a univalent Hofmann-Streicher universe on an
    inaccessible set.
  \item Path types are identity types.
  \end{enumerate}

  Then we deduce the law of excluded middle.
\end{theorem}

\begin{proof}
  By lemma \ref{lem:czfwkstbunit}, any inaccessible set $V$ has a
  weakly $\neg \neg$-stable unit. We deduce by lemma
  \ref{lem:negnegunittopointwise} that any homotopical universe on $V$
  has a pointwise weakly $\neg \neg$-stable unit. Hence by
  lemma \ref{lem:negnegunittocof} we see that if $m$ is a monomorphism
  in $\set$ and $\Delta(m)$ is a cofibration then $m$ is
  $\neg \neg$-stable. However, by the assumption that path types are
  identity types and theorem \ref{thm:allmonoscofs}, all monomorphisms
  are cofibrations. Applying this to $\Delta(m)$ where $m$ is any
  monomorphism in $\set$, we deduce that every monomorphism in $\set$
  is $\neg \neg$-stable. Excluded middle follows.
\end{proof}

We now consider internal categories in realizability toposes. We first
note that realizability toposes are never boolean (except for the
trivial case).

\begin{lemma}
  \label{lem:rtnotboolean}
  Suppose that $\pcaa$ contains two distinct elements $x \neq y$ (we
  say $\pcaa$ is \emph{non trivial}). Then $\rt(\pcaa)$ is not a
  boolean topos.
\end{lemma}

\begin{proof}
  We need to show that $\top \colon 1 \to 2$ is not a subobject
  classifier, so it suffices to find a monomorphism that is not a
  pullback of $1 \to 2$. We take this monomorphism to be the canonical
  map $2 \to \nabla 2$.

  Since $\pcaa$ is non trivial we have
  $\underline{0} \neq \underline{1}$, by a similar proof to
  \cite[Proposition 1.3.1, part iii]{vanoosten}. Therefore all maps
  from $\nabla 2$ to $2$ are constant, and clearly the map $2 \to
  \nabla 2$ is not the
  pullback along either of the constant maps.
\end{proof}

\begin{theorem}
  There is no category $\catc$ satisfying all of the following
  conditions.
  \begin{enumerate}
  \item $\catc$ is a category of internal presheaves over an internal
    category $\smcat{C}$ in a realizability topos $\rt(\pcaa)$ where
    $\pcaa$ is non trivial.
  \item $\smcat{C}$ has finite products.
  \item $\catc$ has a class of cofibrations and interval object
    satisfying our general conditions.
  \item The interval object is connected.
  \item $\catc$ has a univalent homotopical Hofmann-Streicher universe
    on a McCarty universe $\mc(\pcaa)$.
  \item Path types are identity types.
  \end{enumerate}
\end{theorem}

\begin{proof}
  By lemma \ref{lem:mccartyhasnegnegunit}, any McCarty universe
  $\mc(\pcaa)$ has a weakly $\neg \neg$-stable unit. Hence any
  homotopical Hofmann-Streicher universe on $\mc(\pcaa)$ has a
  pointwise weakly $\neg \neg$-stable unit by lemma
  \ref{lem:negnegunittopointwise}. We deduce by lemma
  \ref{lem:negnegunittocof} that for every monomorphism $m$ in
  $\rt(\pcaa)$, if $\Delta(m)$ is a cofibration then $m$ is
  $\neg \neg$-stable. However, by the assumption that path types are
  identity types and theorem \ref{thm:allmonoscofs} all monomorphisms
  are cofibrations. Hence all monomorphisms in $\rt(\pcaa)$ are
  $\neg \neg$-stable, including the subobject classifier
  $1 \to \Omega$, and so we deduce the law of excluded middle,
  contradicting lemma \ref{lem:rtnotboolean}.
\end{proof}

We observe that all of the results for internal presheaves apply in
particular to the degenerate case where the internal category is
trivial. In this case, for instance, pointwise $\neg \neg$-separated
is the same as $\neg \neg$-separated, whereas in general it is usually
weaker. Also note that in this case the Hofmann-Streicher universe on
$V$ is $V$ itself. We will apply this to two realizability toposes in
particular: the effective topos and the Kleene-Vesley topos. These
were studied from a homotopical point of view respectively by Van den
Berg and Frumin in \cite{vdbergfrumin} and the author in \cite[Section
8.2]{swanwtypered}. In both cases our argument depends only on the
choice of interval object and is independent of the choice of
cofibrations, as long as they satisfy our general conditions. In
particular for these results we don't need to assume path types are
identity types.

Recall (from \cite{vdbergfrumin}) that we can define an interval
object in the effective topos on $\nabla 2$, the uniform object with
$2$ elements.

\begin{theorem}
  Suppose we are given a class of cofibrations in the effective topos,
  that together with the interval object $\nabla 2$ satisfies our
  general conditions. Then there is no univalent homotopical universe
  on a McCarty universe, and no $\neg \neg$-separated univalent
  universe closed under propositional truncation.
\end{theorem}

\begin{proof}
  Note that any map $1 \to 2^\nat$ can be viewed as a pullback of an
  endpoint inclusion $1 \to \nabla 2$. We can therefore use lemma
  \ref{lem:negnegunittocof} together with the same argument as in
  theorem \ref{thm:klassembliesnopathid}.
\end{proof}

We recall that countably based $T_0$-spaces embed into the function
realizabilty topos, as shown by Bauer \cite{bauereqsptte}. The
subcategory $\kv$ consists of maps that are both computable and
continuous and hence one can view the usual topological interval as an
interval object in $\kv$. It is straightforward to find a connection
structure for the interval using the usual topological definitions. It
is currently unclear what the best choice of cofibration for $\kv$ is,
but the theorem below applies in any case, as long as our general
conditions are satisfied. We note that there is at least one non
trivial example given by taking all monomorphisms to be cofibrations.

\begin{theorem}
  Suppose we are given a class of cofibrations in the Kleene-Vesley
  topos, that together with the topological interval object $[0, 1]$
  satisfies our general conditions. Then there is no univalent
  universe on a McCarty universe, and no $\neg \neg$-separated
  univalent universe closed under propositional truncation.
\end{theorem}

\begin{proof}
  We first show that the map $1 \to 2^\nat$ defined to be constantly
  $\lambda n.0$ can be viewed as a pullback of the endpoint inclusion
  $\delta_0 \colon 1 \to [0, 1]$. We define a continuous map
  $h \colon 2^\nat \to [0, 1]$ by taking $h(\alpha)$ to be
  $\Sigma_{n = 0}^\infty \,2^{-\alpha(n)}$. This is evidently
  computable, and so does define a map in $\kv$, and it is
  straightforward to check that $\lambda n.0$ is the pullback of the
  $\delta_0$ along $h$.

  We can therefore
  use lemma \ref{lem:negnegunittocof} together with the same argument
  as in theorem \ref{thm:klassembliesnopathid}, and the observation
  that $\llpo$ fails in $\kv$.
\end{proof}

Both of the above results are specific to McCarty universes and
$\neg \neg$-separated universes. This of course leaves open
the possibility of constructing univalent universes in a completely
different way. One such possibility is to find a constructive version
of the definition by Shulman in \cite[Section 3]{shulmanelegantreedy}.

\section{Cofibrations in Homotopy Type Theory}
\label{sec:typetheory}

\newcommand{\fibre}[2]{\Sigma_{x : A} #1 = #2}

Although technically the results so far are specific to Orton-Pitts
models of type theory, they illustrate ideas that might turn out to be
more widely applicable. In particular the use of
$\neg \neg$-separation in presheaf assemblies matches up well with
definitional equality in type theory. In assemblies the
$\neg \neg$-stable propositions are those that are ``free of
computational information.'' That is, we don't need to be told a
particular realizer to show they are realized; we can guess a realizer
uniformly and if they are true then the realizer works. Meanwhile
equality in the underlying category is used to implement definitional
equality in type theory. This is something that should be ``free of
computational information,'' in the sense that we don't need a proof
term to exist to show two terms are definitionally equal. We are (or
perhaps should be) able to work out whether two terms are are equal or
not just by looking at the terms themselves without being given any
extra computational information.

We will use this idea to show that there is a purely type theoretic
version of the construction used in lemma \ref{lem:vsephs}. We will
use this to give interesting proofs of some minor results regarding
the syntactic category of homotopy type theory. This also suggests
that in future work it may be possible to obtain a syntactic version of
the main theorem (as will be discussed further in the conclusion).


We work over homotopy type theory as defined in
\cite{hottbook}. Recall that we can use the syntax of type theory to
define a category that is referred to as the \emph{syntactic category}
or \emph{classifying category}.  As shown by Lumsdaine in
\cite{lumsdainecofibtt}, building on the results of Gambino and Garner
\cite{gambinogarner}, one can define a notion of cofibration in the
syntactic category of type theory, and in the presence of suitable
higher inductive types one in fact gets a model structure. We define
display maps to be maps of the form $\Gamma.X \to \Gamma$. We then
define trivial cofibrations to be maps with the left lifting property
against display maps and cofibrations to be maps with the left lifting
property against those display maps $\Gamma.X \to \Gamma$ where there
is a term witnessing that $X$ is contractible.

Following the notation in \cite{hottbook} we write
$U$ for the universe of small types and $\ast$ for the unique element
of the unit type $1 : U$. We will follow the convention of using the
half adjoint definition of equivalence. Hence a term witnessing the
equivalence is a tuple containing maps in both directions as
components (in addition to the terms witnessing that the maps
are mutually inverse and the half adjoint coherence term).

In order to state the results that follow, we define the notion of
canonical map in the syntactic category type theory.
\begin{definition}
  We say a map $\sigma \colon \Delta \to \Gamma$ in the syntactic
  category is \emph{canonical} if for every map
  $\tau \colon \Xi \to \Gamma$ we can effectively decide whether or
  not there exists a map $\mu$ in the diagram below.
  \begin{equation*}
    \xymatrix{\Xi \ar@{.>}[d]_\mu \ar[dr]^\tau & \\
      \Delta \ar[r]_\sigma & \Gamma}
  \end{equation*}
  Moreover, if $\mu$ exists then it is unique and we can find it
  effectively.
\end{definition}

\begin{proposition}
  \label{prop:canontomonic}
  If $\sigma \colon \Delta \to \Gamma$ is canonical then it is a
  monomorphism.
\end{proposition}

\begin{proof}
  This follows from the uniqueness condition in the definition of
  canonical.
\end{proof}

We now give our syntactic version of lemma \ref{lem:vsephs}.
\begin{lemma}
  \label{lem:syntaxlem}
  Let $m \colon \Gamma.A \to \Gamma.B$ be a cofibration over
  $\Gamma$. Then there is a raw term $r$ whose only free variables
  either belong to $\Gamma$ or are equal to a fresh variable $y$ such
  that for any context $\Delta$, any $\sigma \colon \Delta \to \Gamma$
  and any terms $a$ and $b$ with $\Delta \vdash a : A[\sigma]$,
  $\Delta \vdash b : B[\sigma]$ and $\Delta \vdash m[\sigma][x/a]
  \equiv b$, we have the following.
  \begin{align}
    \Delta &\vdash r[\sigma][y/b] : \| \Sigma_{x : A[\sigma]} \, m =
             b \| \label{eq:syntaxlemvalid} \\
    \Delta &\vdash r[\sigma][y/b] \equiv | (a, \refl_{m[x/a]})
             | \label{eq:syntaxlemeq}
  \end{align}
\end{lemma}

\begin{proof}
  Write $C$ for the type $\|\Sigma_{x : A} m = y \|$. Then we
  have the valid judgement $y : B \vdash C(y) : U$.  Note that if we
  reindex along $m$, we get the type $x : A \vdash C[y/m] : U$. In
  any case $C$ is an hproposition, and in context
  $x : A$ we can clearly construct an inhabitant
  $| (x, \refl_{m(x)}) |$ of $C[y/m(x)]$. Hence in context $x : A$,
  $C[y/m]$ is contractible and so we can construct a term witnessing
  that $C[y/m]$ is equivalent to the unit type $1 : U$. We take the
  component witnessing the map $1 \to C[y/m]$ to be $\lambda z.(x,
  \refl_{m(x)})$. We omit writing out the other components of the
  equivalence.

  Since $m$ is a cofibration by assumption, we can deduce by
  univalence that there is a type $D$ with
  $\Gamma, y : B \vdash D : U$ together with a term witnessing that
  $D$ is equivalent to $C$ in context $\Gamma, y :
  B$. Explicitly, we use a diagonal filler in the lifting problem
  below.
  \begin{equation*}
    \xymatrix@C=10em{ \Gamma.A \ar[r]^(0.35){(C, \lambda x.1, \lambda
        z.|(x, \refl_{m(x)})|)} \ar[d] &
      \Sigma_{C : U} \Sigma_{D : U} \, \hequiv(D, C) \ar[d] \\
      \Gamma.B \ar[r]_C \ar@{.>}[ur] & U}
  \end{equation*}
  We write
  $e$ for the component of the equivalence that witnesses the map
  $D \to C$ and omit writing the other components of the
  equivalence. We will take $r$ to be $e \ast$. The upper
  triangle law for the diagonal filler tells us that when we reindex
  along $m$ we get the following definitional equalities.
  \begin{align*}
    \Gamma, x : A &\vdash D[y/m] \equiv 1 : U \\
    \Gamma, x : A &\vdash e[y/m] \equiv \lambda z.|(x, \refl_{m(x)})| : 1 \to
  C[y/m]
  \end{align*}

  Substituting in $\sigma$ and $a$ we deduce the following.
  \begin{align}
    \Delta &\vdash D[\sigma][y/m[\sigma][x/a]] \equiv 1 \label{eq:4}\\
    \Delta &\vdash e[\sigma][y/m[\sigma][x/a]] \equiv \lambda z.|(a,
             \refl_{m[x/a]})| \label{eq:5}
  \end{align}
  Using $\Delta \vdash m[\sigma][x/a] \equiv b$ we deduce the following.
  \begin{align}
    \Delta &\vdash D[\sigma][y/m[\sigma][x/a]] \equiv
             D[\sigma][y/b] \label{eq:1} \\
    \Delta &\vdash e[\sigma][y/m[\sigma][x/a]] \equiv
             e[\sigma][y/b] \label{eq:2}
  \end{align}
  We then combine \eqref{eq:4} with \eqref{eq:1} and \eqref{eq:5} with
  \eqref{eq:2} to get the following.
  \begin{align}
    \Delta &\vdash D[\sigma][y/b] \equiv 1 \label{eq:6}\\
    \Delta &\vdash e[\sigma][y/b] \equiv \lambda z.|(a,
             \refl_{m[x/a]})| \label{eq:7}
  \end{align}

  From \eqref{eq:6} we derive $\Delta \vdash \ast : D[\sigma][y/b]$,
  and so we can derive the following judgement.
  \begin{equation*}
    \Delta \vdash e[\sigma][y/b]\ast : C[\sigma][y/b]
  \end{equation*}
  But this is the same as \eqref{eq:syntaxlemvalid}.

  From \eqref{eq:7} we can derive the following.
  \begin{equation*}
    \Delta \vdash e[\sigma][y/b]\ast \equiv |(a, \refl_{m[x/a]})| : C[y/b]
  \end{equation*}
  But this is the same as \eqref{eq:syntaxlemeq}.
\end{proof}

\begin{lemma}
  \label{lem:trunccanontocanon}
  Suppose that we are given types $A$ and $B$ in a context $\Gamma$
  and a term $\Gamma, x : A \vdash m : B$. Suppose that the truncation
  map $(\Gamma, y : B, \fibre{m}{y}) \to (\Gamma, y : B, \|
  \fibre{m}{y} \|)$ is canonical and that we have decidable type
  checking. Then the map $(1_\Gamma, m) \colon \Gamma.A \to \Gamma.B$
  is canonical.
\end{lemma}

\begin{proof}
  Suppose we are given a map $\tau \colon \Xi \to \Gamma.B$. Note that
  we can split up $\tau$ as $(\sigma, b)$ where
  $\sigma \colon \Xi \to \Gamma$ and $\Xi \vdash b : B[\sigma]$.

  Let $r$ be a raw term as in the statement of lemma
  \ref{lem:syntaxlem}. We first use decidable type checking to decide
  whether the following judgement is valid.
  \begin{equation*}
    \Xi \vdash r[\sigma][y/b] : \| \fibre{m[\sigma]}{b} \|
  \end{equation*}
  If it is not valid we say there is no such term $a$ satisfying the
  condition. If it is valid, we continue.
  
  Now using the assumption that the truncation map is canonical, we
  can effectively decide whether or not there exists a term $c$
  satisfying the following.
  \begin{align*}
    \Xi &\vdash c : \fibre{m[\sigma]}{b} \\
    \Xi &\vdash r[\sigma][y/b] \equiv |c|
  \end{align*}
  If the check returns false we say there is no such term $a$
  satisfying the condition. If it is valid, then we can effectively
  find such a term $c$, and we continue.

  We next use type checking to decide if the following judgement is
  valid.
  \begin{equation*}
    \Xi \vdash c \equiv (\pi_0 c, \refl_{m[\sigma][x/\pi_0 c]})
  \end{equation*}
  If so, then we have found a suitable term taking $a := \pi_0 c$,
  otherwise we say there is no such term.

  We now need to show that the term $a$ is unique and that if any of
  the three checks above returns false then there really is no such
  term $a$. It suffices for both to show that if $a$ is any such term
  then all the checks above return true, and that for the resulting
  term $c$ we have $\Xi \vdash a \equiv \pi_0 c$. So, let $a$ be
  any term such that $\Xi \vdash a : A$ and
  $\Xi \vdash m[\sigma][x/a] \equiv b$.

  Lemma \ref{lem:syntaxlem} tells us that we have the judgements
  below.
  \begin{align*}
    \Xi &\vdash r[\sigma][y/b] : \| \fibre{m[\sigma]}{b} \|  \\
    \Xi &\vdash r[\sigma][y/b] \equiv | (a, \refl_{m[\sigma][x/a]})
         |
  \end{align*}
  Then the first judgement tells us that the first type check must
  have returned true. Next, the two judgements together with
  canonicity for truncation tell us that the second test must have
  returned true, and that for the resulting term $c$ we have
  $\Xi \vdash c \equiv (a, \refl_{m[\sigma][x/a]})$.

  We can now deduce that $\Xi \vdash \pi_0 c \equiv a$, and that
  the final type check must have also returned true, as required.
\end{proof}

We will now use the lemma to prove a couple of minor results about
canonical maps. In each case, the result itself isn't so interesting
so much as that we can prove them without using strong normalisation,
or something similar.

The first result is analogous to the kind of construction that was
very useful when we were working semantically. Unfortunately,
it is currently unclear if there are any new non trivial examples of
applications when working syntactically. However, we do have the minor
observation that coproduct inclusions are monic, in the sense that if
$\Gamma \vdash a, a' : A$ and $\Gamma \vdash \mathtt{inl}(a) \equiv
\mathtt{inl}(a') : A + B$ then $\Gamma \vdash a \equiv a' : A$.

\begin{theorem}
  Suppose that we have decidable type checking. Suppose that $A$ and
  $B$ are types in context $\Gamma$, $m$ is a term
  $\Gamma, x : A \vdash m : B$, that we are given a term witnessing
  that $m$ is an embedding and that
  $(1_\Gamma, m) \colon \Gamma.A \to \Gamma.B$ is a cofibration. Then
  $(1_\Gamma, m)$ is canonical.
\end{theorem}
\begin{proof}
  Since $m$ is an embedding, by definition the type
  $\Sigma_{x : A}\,m = y$ is an hproposition in context
  $\Gamma, y : B$. It easily follows that the truncation map
  $\Gamma.B.\Sigma_{x : A}\,m = y \to \Gamma.B.\|\Sigma_{x : A}\,m =
  y\|$ has a retraction, and hence is canonical. We can now apply
  lemma \ref{lem:trunccanontocanon}.
\end{proof}

\begin{theorem}
  \label{thm:trunccanontocanonclosed}
  Suppose that for all closed types $A$ and $B$ and maps between
  singleton contexts $m \colon (A) \to (B)$, the truncation map
  $(y : B, \fibre{m}{y}) \to (y : B, \| \fibre{m}{y} \|)$ is
  canonical, and that we have decidable type checking.

  Then for any closed types $A$ and $B$, any cofibration
  $m \colon (A) \to (B)$ is canonical.
\end{theorem}

\begin{proof}
  This is a special case of lemma \ref{lem:trunccanontocanon} where we
  take $\Gamma$ to be empty.
\end{proof}

Similar results can be obtained from a well known theorem by Nicolai
Kraus \cite[Section 8.4]{kectanonex}. By a similar (but easier)
argument to theorem \ref{lem:trunccanontocanon} one can use Kraus'
result to show that if $A$ is an inhabited transitive type and
decidable type checking holds, then the truncation map
$\Gamma.A \to \Gamma.\|A\|$ is canonical. It is straightforward to
check that truncation maps are cofibrations, but it is also an
instance of a general principle by Lumsdaine \cite{lumsdainecofibtt},
stating that point constructors of higher inductive types are always
cofibrations.

In fact the proof Kraus used applies not just to truncation maps, but
to any cofibration, as long as the domain satisfies the requirement of
having terms witnessing it is transitive and inhabited. Hence a
cofibration $m \colon \Gamma.A \to \Gamma.B$ is canonical whenever $A$
is transitive and inhabited, and in fact it follows that a cofibration
is monic whenever $A$ is transitive (but not necessarily
inhabited). For example, when $A$ is transitive and $R$ is a binary
relation on $A$, the set quotient map $A \to A/R$ is always monic (but
obviously not always ``homotopy monic'').


\section{Conclusion}
\label{sec:conclusion}

\subsection{Towards a Proof that Path Types are not Identity Types}

The results here and in particular section \ref{sec:typetheory},
suggest that similar results might hold in general in type theory
with univalence. Roughly speaking I expect that in any type theory
with a notion of path type that behaves similar to an exponential, it is
impossible to simultaneously satisfy all three of the following
requirements.
\begin{enumerate}
\item \label{condspathid}Path types are definitionally isomorphic to
  identity types.
\item \label{condshott}Univalence and all the higher inductive
  types defined in \cite{hottbook} are derivable.
\item \label{condscomputes}The type theory has good computational
  properties such as strong normalisation, decidable type checking and
  canonicity.
\end{enumerate}

Unfortunately, even formulating this statement precisely is a
difficult task. For example, to even give a good definition of what a
type theory is in general remains an area of active research. For this
reason this doesn't deserve to be called a ``conjecture;'' sometimes
the term ``hypothesis'' is used for such statements.

To be clear, even if the hypothesis is correct, it allows for
consistent type
theories where any two of the three conditions are satisfied.

For example, cubical type theory as appears in \cite{coquandcubicaltt}
would be an example of a type theory satisfying \ref{condshott} and
\ref{condscomputes}. In \cite{hubercanonicity}, Huber showed that
cubical type theory does satisfy canonicity and suggests that the
technique could be extended to also show the other good computational
properties hold.

Earlier versions of cubical type theory that feature the regularity
condition are likely examples of theories satisfying \ref{condspathid}
and \ref{condscomputes}.

One approach to obtaining a theory satisfying conditions
\ref{condspathid} and \ref{condshott} is to build on work by Isaev in
\cite{isaevmodelstr}. This contained a definition of a type theory
$\mathrm{coe}_1 + \sigma + \mathrm{Path} + \mathrm{wUA}$, with a built
in notion of path type and coercion where coercion
satisfies a computation rule denoted $\sigma$, akin to
regularity, that allows one to implement identity types as path
types. It also satisfies a weak (but computationally meaningful)
version of univalence, denoted $\mathrm{wUA}$. However, no claim is
made regarding decidability of type checking, canonicity or strong
normalisation.

\subsection{Towards a Proof that Path Types are Identity Types}
\label{sec:towards-proof-that}

Although the aim of this paper was towards finding counterexamples, we
note that the hypothesis in the previous section remains just a
hypothesis and so could easily be false.

In particular, for each of the three examples satisfying only two of
the conditions, there is the possibility that the hypothesis can be
falsified by showing that in fact the third remaining condition does
hold. Namely, one could show that the hypothesis is false through any
of the following.
\begin{enumerate}
\item Showing that cubical type theory can be extended with extra
  computational rules that allow us to use path types as identity types,
  while retaining its good computational properties.
\item Showing that in fact it is possible to construct a univalent
  universe in cubical type theory with a regularity axiom.
\item Showing that Isaev's
  $\mathrm{coe}_1 + \sigma + \mathrm{Path} + \mathrm{wUA}$ does have
  decidable type checking, strong normalisation and canonicity, and
  moreover this remains true if it is extended with a universe
  satisfying full univalence and with higher inductive types.
\end{enumerate}

\bibliographystyle{abbrv}
\bibliography{mybib}{}

\end{document}